\documentclass[11pt,a4paper]{article}
\usepackage{graphicx}
\usepackage[english]{babel}
\usepackage{amsthm,amsfonts,amssymb}
\usepackage{amsrefs}
\usepackage{amsmath}
\usepackage{bm}
\usepackage{multirow}

\usepackage{epsfig,amsfonts,amsthm,amssymb,latexsym,amsmath}

\theoremstyle{plain}
\newtheorem{theorem}{Theorem}[section]
\newtheorem{proposition}[theorem]{Proposition}

\newtheorem{lemma}[theorem]{Lemma}

\theoremstyle{remark}
\newtheorem{remark}[theorem]{Remark}

\theoremstyle{definition}

\title{On sequences of projections of the cubic lattice}
\vspace{5pt}
\author{Antonio Campello  \\ \small{Institute of Mathematics, Statistics} \\ \small{ and Computer Science} \\  \small{University of Campinas} \\ \small{ 13083-859, Campinas-SP, Brazil} \\
            \small{campello@ime.unicamp.br}       \and
        Jo\~ao Strapasson \\ \small{School of Applied Sciences} \\ \small{University of Campinas} \\ \small{13484-350, Limeira-SP, Brazil} \\
            \small{joao.strapasson@fca.unicamp.br}
}
\date{}

\begin{document}
\maketitle


\begin{abstract}
In this paper we study sequences of lattices which are, up to similarity, projections of $\mathbb{Z}^{n+1}$ onto hyperplanes $\bm{v}^{\perp}$, with $\bm{v} \in \mathbb{Z}^{n+1}$. We show a sufficient condition to construct sequences converging at rate $O(1/ \left\| \bm{v} \right\|^{2/n})$ to integer lattices and exhibit explicit constructions for some important families of lattices. The problem addressed here arises from a question of communication theory.
\end{abstract}

\section{Introduction}
It was recently proved \cite{Sloane1} that any $n$-dimensional lattice can be approximated by a sequence of lattices such that each element is, up to similarity, the orthogonal projection of the cubic lattice $\mathbb{Z}^{n+1}$ onto a hyperplane determined by a linear equation with integer coefficients. Given a target lattice $\Lambda \subset \mathbb{R}^n$,  it is possible to find a vector $\bm{v} \in \mathbb{Z}^{n+1}$ from the construction in \cite{Sloane1}, such that the distance between $\Lambda$ and a lattice which is equivalent to the projection of $\mathbb{Z}^{n+1}$ onto $\bm{v}^{\perp}$ has order $O(1 / \left\| \bm{v} \right\|^{1/n})$, where $\left\| \bm{v} \right\|$ is the Euclidean norm of $\bm{v}$. A natural question that arises from that result is whether it is possible to improve this convergence. We give a positive answer to this question by showing a sufficient condition to obtain sequences converging to an integer lattice with order $O(1 / \left\| \bm{v} \right\|^{2/n})$. We also show explicit constructions of such sequences for some families of lattices ($D_n$, odd $n$, $D_n^*$) and exhibit a table of which is, to our knowledge, the best sequences of projection lattices in the sense of the tradeoff between density and $\left\| \bm{v} \right\|$.

Apart from the purely geometric interest, the problem of finding sequences of projection lattices with a better order of convergence is motivated by an application in joint source-channel coding of a Gaussian channel \cite{Sueli}. In the aforementioned paper, the authors propose a coding scheme based on curves on flat tori and show that the efficiency of this scheme is closely related to the ``small-ball radius'' of these curves, which can be approximated by the packing radius of a lattice obtained by projecting $\mathbb{Z}^{n+1}$ onto the subspace $\bm{v}^{\perp}$ for $\bm{v} \in \mathbb{Z}^{n+1}$. Given a value $l_0 > 0$, a worth objective to the design of good codes in the sense of \cite{Sueli} is the one of choosing a vector $\bm{v} \in \mathbb{Z}^{n+1}$ with $\left\| \bm{v} \right\| = l_0$ in such a way to maximize

\begin{equation}
r(\bm{v}) = \min_{\bm{n} \in \mathbb{Z}^{n+1}} \min_{t \in \mathbb{R}} \left\| \bm{v}t - \bm{n} \right\|,
\end{equation}
which is the length of the shortest vector of $\Lambda_{v}$, the projection of $\mathbb{Z}^{n+1}$ onto $\bm{v}^{\perp}$. Let $\delta_{\Lambda_{v}}$ be the center density of these lattices (for undefined terms see Section II). Since the volume of $\Lambda_v$ is given by $1/ \left\| \bm{v} \right\|$ (see \cite{FatStrut}), we have:

\begin{equation}
\delta_{\Lambda_{\bm{v}}} = \frac{r(\bm{v})^{n} \left\| \bm{v} \right\|}{2^{n}},
\end{equation}
therefore maximizing $r(\bm{v})$ implies maximizing $\delta_{\Lambda_{\bm{v}}}$.

Another geometrical formulation to this problem is the so-called \textit{fat strut problem}. A ``strut'' is defined as a cylinder anchored at two points in $\mathbb{Z}^{n+1}$ such that its interior does not contain any other integer point. Given $l_0 > 0$, the fat strut problem asks for a vector $\bm{v} \in \mathbb{Z}^{n+1}$ of length $l_0$ that maximizes the radius of the strut anchored at $\bm{0}$ and $\bm{v}$. This problem is shown to be equivalent to the one of finding dense projections of $\mathbb{Z}^{n+1}$ \cite{FatStrut}. Therefore, projection lattices with higher densities imply fat-struts with larger radii, and the problem addressed in this work is related to finding small vectors that attain high density projection lattices. This is done by considering families of projections of $\mathbb{Z}^{n+1}$.

This paper is organized as follows. In Section 2 we summarize some relevant concepts and results on lattices. In Section 3, we derive a sufficient condition to construct good sequences of projection lattices and in Section 4 we exhibit explicit constructions for some well-known lattices. Finally, in Section 5 we present our conclusions.

\section{Preliminaries and Notation}
In this section we give a brief review of some relevant concepts concerning lattices and establish the notation to be used from now on. 

Given $m$ linearly independent vectors $\bm{b}_1,\ldots,\bm{b}_m$ in $\mathbb{R}^n$, a \textit{lattice} $\Lambda$ is the set of all integer linear combinations of these vectors. The matrix $G$ whose rows are the vectors $\bm{b}_i$ is called a \textit{generator matrix} for $\Lambda$ and the matrix $A = G G^t $ is said to be a \textit{Gram matrix} for $\Lambda$. The \textit{determinant} or \textit{discriminant} of $\Lambda$ is defined as $\det \Lambda = \det A$ and corresponds to the square of the volume of any fundamental region for the lattice $\Lambda$. We say that two lattices with generator matrices $G_1$ and $G_2$ are equivalent if there exists an unimodular matrix $U$, an orthogonal matrix $Q$ and a real number $c$ such that $G_1 = c \,\, U \,\, G_2 \,\, Q$. The \textit{density} $\Delta$ of a lattice is the ratio between the volume of a sphere of radius $\rho$ (half of the minimal distance between two distinct lattices points) and the volume of a fundamental region, while the \textit{center density} is defined as $\delta = \Delta/ V_n$ where $V_n$ is the volume of the unitary sphere in $\mathbb{R}^n$. Sometimes we will refer to the center density of a specific lattice $\Lambda$ as $\delta_{\Lambda}$.

Let $G$ be a full-rank generator matrix for $\Lambda$. The \textit{dual lattice} $\Lambda^*$ of $\Lambda$ is the set of all $\bm{x} \in \mbox{span}(G)$ such that $\langle \bm{x}, \bm{y} \rangle$ is an integer number for all $\bm{y} \in \Lambda$, where $\mbox{span}(G)$ is the row space of $G$. One can easily verify that $(GG^t)^{-1} G$ generates $\Lambda^*$. We say that $\Lambda$ is an \textit{integer} (or \textit{rational}) lattice if its generator matrix has integer (rational) entries. All rational lattices are integers up to scale. The cubic lattice $\mathbb{Z}^n$ is the full-dimensional integer self-dual lattice that has the canonical vectors $e_1 = (1,0,\ldots,0), \ldots, e_n = (0,\ldots,0,1)$ as a basis. A list of the densest known packings in some dimensions as well as many other information about lattices can be found in \cite{SloaneBook}.

We say that a sequence of lattices $\Lambda_w$ converges to $\Lambda$ if there exist Gram matrices $A$ for $\Lambda$ and $A_w$ for $\Lambda_w$ such that $\left\|A_w - A \right\|_{\infty} \rightarrow 0$ as $w \to \infty$ where $\left\| M \right\|_{\infty} = \max_{i,j} \left| M_{ij} \right|$. Of course, if $\Lambda_w$ converges in that sense it also converges in any matrix norm. Another matrix norm we use in this work is the Frobenius norm, given by $\left\| M \right\|_{F} = \sqrt{\mbox{tr}(M M^{t})}$.

Finally, we call $\Lambda_1$ a \textit{projection lattice} of $\Lambda_2$ if it is obtained by projecting $\Lambda_2$ onto the subspace orthogonal to a vector $\bm{v} \in \Lambda_2$, the \textit{projection vector}. In this paper, a projection lattice will always be a projection of the cubic lattice onto $\bm{v}^{\perp}$ for $\bm{v} \in \mathbb{Z}^{n+1}$ a primitive vector (i.e., whose entries have greatest common divisor equal to $1$). In the context of the projection lattices, the results in \cite{FatStrut} and \cite{Sloane1} are remarkable. The first one gives an achievable bound for the density of the projection lattices comparable to the so-called Minkowski-Hlawka bound while the second one states that every lattice can be approximated by a sequence of lattices that are equivalent to projection lattices. More formally, given a $n$-dimensional lattice $\Lambda$, it is shown that for every $\varepsilon > 0$, there is a vector $\bm{v} \in \mathbb{Z}^{n+1}$ and a constant $c$ such that there is a Gram matrix $A_v$ for the lattice obtained by projecting $\mathbb{Z}^{n+1}$ onto $\bm{v}^{\perp}$ and a Gram matrix $A$ for $\Lambda$ satisfying $\left\| A - c A_v \right\| \leq \varepsilon$. In this work we make a slight modification on the construction in \cite{Sloane1} that leads to many other projection lattices sequences converging to a target lattice. We then make an error analysis for these sequences and show a sufficient condition for achieving a faster order of convergence, as well as explicit constructions for some important lattices.
\section{Motivation}
Considering the coding scheme mentioned in Section 1, there are two problems that can arise:

\begin{enumerate}
\item Given a certain radius $r_0$, what is the vector $\bm{v} \in \mathbb{Z}^{n+1}$ that maximizes $\left\| \bm{v} \right\|$ s.t. $r(\bm{v}) = r_0$? Equivalently: given a minimum distance for the projection lattice, what is the vector $\bm{v} \in \mathbb{Z}^{n+1}$ s.t. the projection of $\bm{v} \in \mathbb{Z}^{n+1}$ has maximal density? In this case, we want to solve the following maximization problem:

\begin{equation} \max_{\bm{v} \in \mathbb{Z}^{n+1}} \left\| \bm{v} \right\| \quad \mbox{ subject to } r(\bm{v}) = r_0 \label{eq:MAX1}.\end{equation}

\item Conversely, given a length $l_0$ what is the the vector $\bm{v} \in \mathbb{Z}^{n+1}$ with $\left\| \bm{v} \right\| = l_0$ for which $r(\bm{v})$ is a maximum ? Equivalently, we want to find the solution to:
\begin{equation} \max_{\bm{v} \in \mathbb{Z}^{n+1}} r(\bm{v}) \quad \mbox{ subject to } \left\| \bm{v} \right\| = l_0 \label{eq:MAX2}.\end{equation}
\end{enumerate}

We illustrate these two problems in the case $n = 4$. We run an exhaustive search to solve the maximization problem \eqref{eq:MAX2} for $2 \leq l_0 \leq \sqrt{270478}$ (i.e., fixing $\left\| \bm{v} \right\| = l_0$). Some examples are illustrated in Table \ref{tabela1}. From this table, we can guess a good solution for the problem \eqref{eq:MAX1}, where $r_0$ is fixed.

For instance, take $r_0 = 0.16385$. The fourth element of the family of projection lattices onto $(1, 2w^2 - w + 1, 2w^2 + w + 1, 4w^2 + 3w)^{\perp}$ has minimal distance $0.163858$, center density $0.164452$ and squared norm $89425$. Among the vectors of similar norm, we can find similar performances (e.g., the lattice produced by the vector $(1, 157, 164, 195)$ has minimal distance, center density and squared norm equal to $ 0.16386$, $0.164594$ and $89571$ respectively).

\begin{footnotesize}
\begin{table}[h]
\label{tabela1}
$$\begin{array}{|c||c|c|c|}
\hline
\mbox{projection vector } \bm{v} &  \mbox{Center density} & \mbox{Minimal norm} &  \left\| \bm{v} \right\|_2^2 \\
\hline
 (1,29,37,268) & 0.173511 & 0.172147 & 74035 \\
\hline
 (1,56,185,196) & 0.16502 & 0.168637 & 75778 \\
\hline
 (1,121,163,187) & 0.170589 & 0.170362 & 76180 \\
\hline
 (1,33,80,265) & 0.16473 & 0.16783 & 77715 \\
\hline
 (1,98,125,230) & 0.168027 & 0.168793 & 78130 \\
\hline
 (1,107,141,222) & 0.166704 & 0.167472 & 80615 \\
\hline
 (1,42,181,215) & 0.166423 & 0.167331 & 80751 \\
\hline
 (1,8,110,265) & 0.165716 & 0.166535 & 82390 \\
\hline
 (1,12,84,282) & 0.164198 & 0.164612 & 86725 \\
\hline
 (1,91,153,236) & 0.166189 & 0.165065 & 87387 \\
\hline
 (1,119,152,224) & 0.165562 & 0.16484 & 87442 \\
\hline
 (1,88,121,256) & 0.16497 & 0.164493 & 87922 \\
\hline
 (1,8,64,292) & 0.164452 & 0.163858 & 89425 \\
\hline
\end{array}$$
\caption{Dense projection lattices of $\mathbb{Z}^4$ onto $\bm{v}^{\perp}$, for $74035 \leq \left\| \bm{v} \right\| \leq 89425$.}
\end{table}
\end{footnotesize}

Now, take $r_0 = 0.1721$ and the fourth element of the sequence of projections determined by the vectors $(1, 2w^2-w+1, 2w^2+w+1, 4w^3+3w)$ (which is an ``optimal'' sequence in the sense discussed in this paper). It has minimal distance, center density and squared norm equal to $0.172147$, $0.173511$ and $74035$. Comparing this to the vector $(1, 13, 75, 244)$, which is ``close'' to $(1,29,37,268)$ in the sense of these parameters, we find out that although the last one has a slightly smaller norm ($65331$), its center density is much smaller ($0.163112$). There is clearly a tradeoff between these parameters.
Finally, let us fix the vector length around $\|\bm{v}\|_2^2=89425$. The fourth element of the first family above-cited produces a projection lattice with parameters ``close'' to the ones of the vector $(1,31,38,295)$ (minimal distance, center density and squared norm equal to $0.163988$, $0.164852$ and $89431$).  On the other hand, there is no vector (with norm up to $89425$) producing denser lattices than the fourth element of the optimal sequence, but within the interval $74035$ and $89425$ there are many other vectors that generate lattices with density superior to $0.163858$, as shown in Table 1.

\section{Convergence rate analysis}

Let $\Lambda \subset \mathbb{R}^n$ be a lattice with a $n \times n$ generator matrix $\bar{G}$ and consider $n \times n$ generator matrix $\bar{G}^*$ to $\Lambda^*$. Let $G^{*} = \left[ \bar{G}^{*} \,\,\, 0_{n \times 1} \right]$ and $G = \left[ \bar{G} \,\,\, 0_{n \times 1} \right]$. We define $\Lambda_w^*$, $w \in \mathbb{N}$ as the sequence of $n$-dimensional lattices in $\mathbb{R}^{n+1}$ associated to the generator matrices

\begin{equation}
G_w^{*} = wG^* + P.
\label{eq:1}
\end{equation}
where $P$ is an $n \times (n+1)$ integer matrix which we will call a \textit{perturbation matrix}. The correspondent Gram matrices for $\Lambda_w^*$ are
\begin{equation}
A_w^{*} = w^2 A^* + w(G^*P^t + PG^{*t}) + PP^t \triangleq w^2 A^* + w Q_1 + Q_0,
\label{eq:2}
\end{equation}
where $A^* = G^* G^{*t}$ is a Gram matrix for $\Lambda^*$. We define $H_w = (G_w^*)_{(1,\ldots,n),(2,\ldots,n+1)}$ as the matrix consisting on the last $n$ columns of $G_w^*$. If $G$ is a lower triangular matrix and $P = \left[0_{n \times 1} \,\,\, I_{n \times n} \right]$ it is shown in \cite{Sloane1} that each $\Lambda_w$ (dual of $\Lambda_w^*$) is the projection of $\mathbb{Z}^{n+1}$ onto the subspace orthogonal to some vector $\bm{v} \in \mathbb{Z}^{n+1}$ and the sequence of Gram matrices $(A_w^*)/c = (G_w^*G_w^{t})/c \rightarrow A^*$ as $w \rightarrow \infty$ (hence, $cA_w \rightarrow A$), for $c = w^2$. A natural extension of this result is the following:

\begin{lemma} Let $G_w^*$ (\ref{eq:1}) and $H_w$ be the matrices defined above. If $H_w$ is unimodular for all $w \in \mathbb{N}$ then the lattices $\Lambda_w$ associated to the generator matrices $G_w$ are projection lattices.
\label{genCons}
\end{lemma}

\begin{proof}
Since $H_w$  is unimodular so is its inverse and $\Lambda_w^*$ is also generated by $H_w^{-1} G_w^*$. On the other hand, $H_w^{-1} G_w^* = \left[\hat{\bm{v}}_w \,\,\, I_{n \times n} \right]$ and for the same arguments of \cite{FatStrut}, $\Lambda_w^*$ is the intersection of $\mathbb{Z}^{n+1}$ with the subspace orthogonal to the vector $\bm{v}_w = (1, -\hat{\bm{v}}_w)$, which is the dual of the projection of $\mathbb{Z}^{n+1}$ onto $\bm{v}_w^{\perp}$. 
\end{proof}

In the sequel we will consider the analysis of the convergence order of the sequence $w^2 A_w$ (i.e., the sequence of duals of (\ref{eq:1})). We start by analysing the sequence (\ref{eq:1}). It is straightforward to show that $1/w^2 A_w^* \to A^*$ with order $O(1/w)$ since

$$\left\| A^* - (1/w^2)A_w^* \right\|_{\infty} = \left\| \frac{Q_1}{w} + \frac{Q_0}{w^2} \right\|_{\infty} = O(1/w).$$

If $Q_1 = 0$, we obtain an $O(1/w^2)$ convergence order, as it happens in the example of \cite{Sloane1}, Section 4. More generally, if $Q_1 = \alpha A^*$, we can evaluate $A_w^*$ as follows:

$$A_w^* = w^2 A^* + \alpha w A^*  + Q_0 = A^* \left(w+\frac{\alpha}{2}\right)^2 - \frac{\alpha^2 A^*}{4}+PP^t,$$
therefore the sequence $A_w^*/(w+\alpha/2)^2$ (i.e., taking $c = (w+\alpha/2)^2$) converges to $A_w^*$ with rate $O(1/w^2)$. Nevertheless, our main objective is the analysis of the sequence $\Lambda_w$. In what follows, we will show that the asymptotic behavior of $\Lambda_w$ is essentially the same as $\Lambda_w^*$, although for finite $w$ they may differ. In order to show this, we will need the following lemma.

\begin{lemma}
Let $A_w^{*}$ be Gram matrices for $\Lambda_w^{*}$ as in Equation (\ref{eq:2}). There exists $w_o$ such that, for $w \geq w_o$, the projection lattices $\Lambda_w$ have Gram matrix

\begin{equation}
A_w = \sum_{k=0}^\infty \left[\frac{A}{w^2}\left(-wQ_1 - Q_0\right)\right]^k \frac{A}{w^2}.
\end{equation}
\label{LemaNeumman}
\end{lemma}
\begin{proof}
According to the matrix Neumman series \cite[Ch. 3, Eq. (3.8.3)]{Meyer}:

\begin{equation}
A_w = (A_w^*)^{-1} = (w^2 A^* + wQ_1 + Q_0)^{-1} = \sum_{k=0}^\infty \left[\frac{A}{w^2}\left(-wQ_1 - Q_0\right)\right]^k \frac{A}{w^2}
\end{equation}

\noindent  provided that $\displaystyle \lim_{k\rightarrow \infty} \left\| \left[(w^2 A^*)^{-1} (wQ_1+Q_0) \right]^k \right\| = 0$ for any matrix norm. Since all entries of $(w Q_1 + Q_0)$ have order $O(w)$ and the entries of $(w^2 A^*)^{-1}$ are $O(w^2)$, there exists $w_o$ such that, for $w \geq w_o$, each entry of the matrix $(w^2 A^*)^{-1} (wQ_1+Q_0)$ is arbitrarily close to zero. Taking the matrix power, we can make $\left[ (w^2 A^*)^{-1} (wQ_1+Q_0) \right]^k < \varepsilon$ for any $\varepsilon > 0$ and the result follows. 
\end{proof}
As a consequence of the above lemma, we have:

\begin{eqnarray*} \left\| A - w^2 A_w \right\| &=& \left\|\sum_{k=1}^\infty \left[\frac{A}{w^2}\left(-wQ_1 - Q_0\right)\right]^k A \right\| \approx \\
&\approx& \left\| \frac{A Q_1 A}{w} + \frac{A Q_0 A}{w^2} \right\| = O(1/w).
\end{eqnarray*}
Again, if $Q_1 = \alpha A^*$ ($\Leftrightarrow AQ_1A = \alpha A$), we obtain an $O(1/w^2)$ convergence through the evaluation

$$\frac{A}{w^2} - \frac{\alpha A}{w^3} + \frac{A Q_0 A}{w^4} = A \left(\frac{1}{w}- \frac{\alpha}{2w^2}\right)^2 + \frac{\alpha^2 A}{4w^4} + \frac{A Q_0 A}{w^4}$$
so that the distance from  $A_w/\left(1/w- \alpha/2w^2\right)^2$ to $A$ has order $O(1/w^2)$.
\begin{remark}
Since  $\displaystyle \lim_{w\rightarrow \infty} \left(\frac{1}{w}-\frac{\alpha}{2 w^2}\right)^2 \left(\frac{\alpha}{2}+w\right)^2 = 1$, we have:

$$\left(A_w/\left(1/w- \alpha/2w^2\right)^2\right)^{-1} \approx A_w^*/(w+\alpha/2)^2.$$

\end{remark}

We can now prove our main theorem concerning the convergence analysis of projection lattices sequences.

\begin{theorem}
\label{teo1}
Let $\Lambda$ be an $n$-dimensional lattice with generator matrix $\bar{G}$ and Gram matrix $A$ and $\Lambda^* \subseteq \mathbb{Z}^{n}$ its dual with generator and Gram matrices $\bar{G}^*$ and $A^*$ respectively. Now, let $\Lambda_w^*$ be the sequence of lattices with generator matrices given by (\ref{eq:1}) satisfying:

\begin{equation}
\label{cond1}
\det(H_w) = \pm 1, \forall w \in \mathbb{N} \mbox{ and }
\end{equation}
\begin{equation}
\label{cond2}
\exists \alpha \mbox{ such that } Q_1 = \alpha A^*,
\end{equation}
with $H_w$, $A_w^*$, $Q_1$ as previously defined and  $A_w = (A_w^*)^{-1}$ . Then each $\Lambda_w = (\Lambda_w^*)^*$ is a projection lattice of $\mathbb{Z}^{n+1}$ onto the orthogonal subspace of a vector $\bm{v}_w \in \mathbb{Z}^{n+1}$ whose infinity norm satisfies \begin{equation}
\label{eq:dependeRepr}
\left\|\bm{v}_w\right\|_{\infty}= \left| \sqrt{\det{\Lambda}^*}w^n + O(w^{n-1}) \right|
\end{equation} 
for sufficiently large $w$, and there exists a $c_w \in \mathbb{R}$ such that
\begin{equation}
\label{eq:quadratica}
\left\| A - c_w A_w \right\|_{\infty} = O \left( \frac{1}{w^2} \right) = O \left( \frac{1}{{\left\|\bm{v}_w\right\|_{\infty}^{2/n}}} \right) \rightarrow 0 \mbox{, as } w \to \infty.
\end{equation}
\end{theorem}

\begin{proof}
We will first show the validity of the Equation (\ref{eq:dependeRepr}) and then (\ref{eq:quadratica}) will hold for $c_w = \left(1/w- \alpha / 2 w^2\right)^2$ provided Lemma \ref{LemaNeumman} and previous arguments. Let $\bm{\bar{v}}_w$ be the generalized cross product of the rows of $G_w^*$ (see \cite{Spivak}) i.e., $(\bm{\bar{v}}_w)_i = (-1)^{n+i} \,\, | \overline{(G_w^*)_i} |$, where $| \overline{(G_w^*)_i} |$ denotes the determinant of the matrix obtained excluding the $i$-th column of $G_w^*$. According to Lemma \ref{genCons}, the projection vector $\bm{v}_w$ will be given by $\bm{v}_w = (1, -\hat{\bm{v}}_w)$for
$$ \hat{\bm{v}}_w = (H_w)^{-1} (G_w^*)_{1} \Rightarrow (G_w^*)_1 = H_w \hat{\bm{v}}_w = -\sum_{j=2}^{n} (\bm{v}_w)_{j} (G_w^*)_{j}.
$$
Hence:
\begin{equation*}
\begin{split} | \overline{(G_w^*)_i} | &= \det{ \left[ - \sum_{j=2}^{n} (\bm{v}_w)_{j} (G_w^*)_{j} \,\,\, \vert \,\,\, (G_w^*)_{2} \,\,\, \vert \,\,\, \ldots \,\,\, \vert \,\,\,  \widehat{(G_w^*)_{ i}} \,\,\, \vert \,\,\, \ldots \,\,\, \vert \,\,\, (G_w^*)_{n} \right] } \\ & = 
\det{ \left[ -(\bm{v}_w)_{i} (G_w^*)_{i} \,\,\, \vert \,\,\, (G_w^*)_{2} \,\,\, \vert \,\,\, \ldots \,\,\, \vert \,\,\,  \widehat{(G_w^*)_{ i}} \,\,\, \vert \,\,\, \ldots \,\,\, \vert \,\,\, (G_w^*)_{n} \right] } \\
& = (-1)^{i} \left| H_w \right| (\bm{v}_w)_{i} \,\,\, \therefore \,\,\, \bar{\bm{v}}_w = (-1)^{n} \left|{H_w}\right| \bm{v}_w,
\end{split}
\end{equation*}
where $\widehat{(G_w^*)_{ i}}$ means the exclusion of the $i$-th column from the matrix. Thus, up to a change of sign, the projection vector is the cross product of the rows of $G_w^*$. Considering that, it is easy to show that each entry of $\bm{v}_w$ is a polynomial of degree up to $n-1$, excepting for the last one, whose absolute value is $| \det \overline{(G_w^*})_n | $ or $|(\bm{v}_w)_n| = |\det( w \overline{G^*} + \overline{(P)_n})| = |w^n \det{\overline{G^*}} + O(w^{n-1})| $ and this completes the proof. 
\end{proof}

In what follows, we show that the vectors $\bm{v}_w$, the densities of the projection lattices and the convergence rate of the sequence $\left\| A - c_w A_w\right\|_{\infty}$ do not depend on the basis choice for $\Lambda^*$.

\begin{proposition}
\label{eq:MudarDeBase} Let $\bar{G_1}^*$ and $\bar{G_2}^*$ be two generator matrices for $\Lambda^*$ with $A_1^*$ and $A_2^*$ the correspondent Gram matrices. Let $G_1^{*} = \left[ \bar{G_1}^{*} \,\,\, 0_{n \times 1} \right]$ and $G_2^{*} = \left[ \bar{G_2}^{*} \,\,\, 0_{n \times 1} \right]$. Now, take the sequence of lattices $\Lambda_{w,1}^*$ associated to the generator matrices $G_{w,1}^* = wG_1^* + P_1$ and let $H_{w,1} = (G_{w,1}^*)_{(1,\ldots,n),(2,\ldots,n+1)}$ such that conditions (\ref{cond1}) and (\ref{cond2}) hold. There exists $P_2 \in \mathbb{Z}^{n \times (n+1)}$ such that the sequence $G_{w,2}^* = wG_{2}^* + P_2$ satisfies conditions (\ref{cond1}), (\ref{cond2}) and

\begin{equation}
H_{w,1}^{-1} G_1^* = H_{w,2}^{-1} G_2^*.
\label{vIgual}
\end{equation}
\end{proposition} 

\begin{proof}
Since $G_1^*$ and $G_2^*$ generate the same lattice, there exists an unimodular matrix $U$ such that $G_1^* = U G_2^*$. We will show that $P_2 = U^{-1}P_1$ satisfies the three properties above.

\noindent For condition (\ref{cond1}), we have:
\begin{equation}
\begin{split}
&G_2^* P_2^t + P_2^* G_2^{*t} = U^{-1} G_1^* P_1^t U^{-t} + U^{-1} P_2^* G_1^{*t} U^{-t} = \\ 
&= U^{-1}(G_1^* P_1^t + P_1 G_1^*t)U^{-t} = \alpha U^{-1} A_1^* U^{-t} = \alpha A_2^*
\end{split}
\end{equation}
\noindent For condition (\ref{cond2}) and Equation (\ref{vIgual}), just observe that $H_{w,2} = U^{-1} H_{w,1}$. 
\end{proof}

Although the search for good sequences is independent of the basis choice, distinct representations (i.e., geometrically similiar) for the same lattice can yield substantially different sequences in terms of the densities of $\Lambda_w$ and the norm of each $\bm{v}_w$, as shown in Equation (\ref{eq:dependeRepr}) and illustrated in examples 3.1 and 3.4.

When there is no perturbation matrix such that $Q_1 = \alpha A^*$, the convergence of the sequence $A_w/w^2$ is  related to the coefficient $A (Q_1 - \alpha A^*) A$. We can thus try the solution of the problem

$$\min \left\|A(G^*P^t + PG^{*t} - \alpha A^*) A\right\|$$
\begin{equation} \mbox{s. t. } \,\, |\det H_w| = 1\,\,\mbox{, } \forall w \in \mathbb{N}
\label{eq:PNLI}
\end{equation}
$$P \in \mathbb{Z}^{n \times (n+1)}$$
$$\alpha \in \mathbb{Z}$$

\noindent which is a non-linear problem of $n^2 + n$ integer variables. In fact, the constraint $\alpha \in \mathbb{Z}$ can be relaxed to $\alpha \gcd((A^*)_{ij}) \in \mathbb{Z} $ and the complexity of this problem is mainly caused by the constraint $|\det H_w| = 1$. Hence, sometimes it is worth considering1 sub-optimal solutions. One possibility is to take a lower triangular matrix $G^*$ and consider the problem:

$$\min \left\|A(G^*P^t + PG^{*t} - \alpha A^*)A\right\|$$
\begin{equation} \mbox{s. t. } \,\, P \in \mathbb{Z}^{n \times (n+1)}
\label{eq:PNLIrel}
\end{equation}
$$P_{ij} = 1 \mbox{, if } j = i+1$$
$$P_{ij} = 0 \mbox{, if } j > i+1,$$
with $\alpha$ constrained as above. In this case, the perturbation matrix $P$ will have the structure

$$P = \begin{bmatrix}
P_{11} & 1 & 0 & \cdots & 0 & 0\\
P_{21} & P_{22} & 1 & \cdots & 0 & 0\\
\vdots & \vdots & \vdots & \ddots & \vdots & \vdots\\
P_{n1} & P_{n2} & P_{n3} & \cdots & P_{nn} & 1
\end{bmatrix}$$
and we can drop the constraint $| \det H_w | = 1$ out.

\section{Explicit Constructions}
In the following examples we employ different strategies to generate the projection lattices families, depending on the structure of each target lattice and the feasibility of finding integer solutions satisfying conditions \eqref{cond1} and \eqref{cond2} or solving the non-linear problem \eqref{eq:PNLIrel}.

\subsection{The lattice $a\mathbb{Z}\oplus b\mathbb{Z}$}

As a first example, consider the lattice generated by the matrix $\bar{G} = (1/ab)\mbox{diag}(a,b)$, $a,b \neq 0$, a scaled version of the $\mathbb{Z}^2$ lattice. We can assume w.l.o.g. that $(a,b)=1$. As a generator matrix for its dual, we choose $\bar{G}^* = \mbox{diag}(a,b)$, and hence, taking a general perturbation $P$, we have:

\begin{equation} G_w = \left[
\begin{array}{ccc}
 a w+P_{1,1} & P_{1,2} & P_{1,3} \\
 P_{2,1} & b w+P_{2,2} & P_{2,3}
\end{array}
\right]
\end{equation}
In this case, the condition \eqref{cond2} is equivalent to:

\begin{equation}
P_{11} = \frac{a \alpha }{2},P_{21} = -\frac{b P_{12}}{a} \mbox{ and } P_{22} = \frac{\alpha  b}{2}
\end{equation}
Since $a$ and $b$ have no common factors, $\alpha$ must be even ($\alpha = 2 \beta$, for $\beta \in \mathbb{Z}$) and $P_{12} = k a$. Under these conditions, we calculate the determinant of $H_w$: 
\begin{equation}
\det (H_w) = a k P_{23}-b \beta  P_{13}-b w P_{13}
\end{equation}
and condition \eqref{cond1} will be satisfied iff

 \begin{equation}
P_{13} = 0 \mbox{ and } akP_{23} = \pm 1.
\end{equation}

Hence, by Theorem \eqref{teo1}, any lattice of the form $\mathbb{Z}\oplus b\mathbb{Z}$, $b\neq 0$ can be recovered as a sequence \eqref{eq:1} or projection lattices whose order of convergence is $O(1/\left\|\bm{v} \right\|)$, while for $a \neq 1$ it is not possible to find a perturbation matrix such that the hypotheses of Theorem \eqref{teo1} hold.

As an interesting consequence of this fact, it is not possible to ensure conditions \eqref{cond1} and \eqref{cond2} to the Example 1 of \cite{FatStrut} scaled by 1/2 i.e., for $a = 2$ and $b = 1$. Nevertheless, for the equivalent lattice $\Lambda = (1/2) \mathbb{Z}\oplus 2\mathbb{Z}$ and its dual $\Lambda^* = \mathbb{Z} \oplus 2\mathbb{Z}$, the sequence of projection lattices associated to the dual of the lattices generated by the matrices

\begin{equation}
G_w^* = 
\left[
\begin{array}{ccc}
 w+c & 1 & 0 \\
 -2 & 2 w+2 c & 1
\end{array}
\right], c \in \mathbb{Z}
\end{equation}
and projection vectors $v_w = [1, -w-c , 2w^2 + 4wc + 2c^2 + 2]$ converges to $\Lambda$ at rate $O(\left\| \bm{v} \right\|)$.

\subsection{The lattice $D_n$ (for odd $n$)}
A possible $O(1/w^2)$ convergence is shown in \cite{Sloane1} for $n = 3$. We here extend this result for any odd $n$. As the matrix $G^*$ (generator matrix for $2 D_n^*$ \cite{Sloane1} with a zero column added to the right) we choose:

$$G^* = \begin{bmatrix}
2 & 0 & \cdots & 0 & 0 & 0 \\
0 & 2 & \cdots & 0 & 0 & 0 \\
\vdots & \vdots & \ddots & \vdots & \vdots & \vdots \\
0 & 0 & \cdots & 2 & 0 & 0 \\
1 & 1 & \cdots & 1 & 1 & 0
\end{bmatrix}$$

We have then the following proposition:

\begin{proposition} There is a sequence of projection lattices that converges to a lattice which is equivalent to $D_n$, for odd $n$, at the rate $O(1/\left\|\bm{v}\right\|^{2/n})$.
\end{proposition}
\begin{proof}
The proof follows by choosing a suitable perturbation matrix. In this case, take $P$ such that

\begin{equation}P_{ij} = \left\{\begin{array}{lc}
(-1)^{i} & \mbox{ if } j = n - i \mbox{ and } 1 \leq i \leq n \mbox{,}\\
(-1)^{i+1} & \mbox{ if } j = n - i + 2 \mbox{ and } 2 \leq i \leq n-1 \mbox{,} \\
\,\,\,\,\,\, 1 & \mbox{ if } (i,j) \in \left\{ (n-1,n+1), (n,n+1) \right\} \\
\,\,\,\,\,\ 1 & \mbox{ if } (i,j) = (1,n) \\
\,\,\,\,\,\, 0 & \mbox{otherwise}
\end{array} \right.\end{equation}
By direct multiplication, one can prove that $G^*P^t + PG^{*t} = 0$ and by elementary operations on the matrix $H_w$ it is possible to prove that $\det H_w = 1 \,\, \forall w $ provided that $n$ is an odd number, thus ensuring that the hypotheses of Theorem \eqref{teo1} hold and the result follows. 
\end{proof}
To illustrate this example, we exhibit below the matrices $G_w^*$, $A_w^*$ and $P$ as well as the vector $\bm{v}_w$ for $n = 5$.

\begin{equation}G_w^*=\left[
\begin{array}{cccccc}
 2 w & 0 & 0 & -1 & 1 & 0 \\
 0 & 2 w & 1 & 0 & -1 & 0 \\
 0 & -1 & 2 w & 1 & 0 & 0 \\
 1 & 0 & -1 & 2 w & 0 & 1 \\
 w & w & w & w & w & 1
\end{array}
\label{pertDn}
\right]\end{equation}
\\ 
$$\\ A_w^* = \left[
\begin{array}{ccccc}
 4 w^2+2 & -1 & -1 & 0 & 2 w^2 \\
 -1 & 4 w^2+2 & 0 & -1 & 2 w^2 \\
 -1 & 0 & 4 w^2+2 & 0 & 2 w^2 \\
 0 & -1 & 0 & 4 w^2+3 & 2 w^2+1 \\
 2 w^2 & 2 w^2 & 2 w^2 & 2 w^2+1 & 5 w^2+1
\end{array}
\right]\mbox{, }$$ \\ $$P = \left[
\begin{array}{cccccc}
 0 & 0 & 0 & 1 & -1 & 0 \\
 0 & 0 & -1 & 0 & 1 & 0 \\
 0 & 1 & 0 & -1 & 0 & 0 \\
 -1 & 0 & 1 & 0 & 0 & 1 \\
 0 & 0 & 0 & 0 & 0 & 1
\end{array}
\right] \mbox{ and } \bm{v}_w = \left[
\begin{array}{c}
  1 \\
 4 w^3+2 w^2+3 w+1 \\
 -4 w^3+2 w^2-3 w+1 \\
 8 w^4+8 w^2+w+1 \\
 8 w^4+8 w^2-w+1 \\
 16 w^5+20 w^3+5w
\end{array}
\right]$$ \\ 

\begin{remark}
For $n = 3$ and a suitable change of basis, the perturbation given by Equation \eqref{pertDn} is precisely the same as the one described in \cite{Sloane1}, Section 4.
\end{remark}

\subsection{$D_n^*$}
Here is a case where $\alpha \neq 0$ is actually necessary. Let us start with the lattice $D_3^*$. As a basis to $({D_3^{*}})^* = D_3$ we take:

\begin{equation}
\bar{G}^* = \left[
\begin{array}{ccc}
 -2 & 0 & 0 \\
 1 & -1 & 0 \\
 0 & 1 & -1 
\end{array}
\right],
\end{equation}
and we want to find a perturbation matrix $P \in \mathbb{Z}_{3 \times 4}$ in order to ensure conditions \eqref{cond1} and \eqref{cond2} of Theorem \eqref{teo1}. Again, starting with a general perturbation, condition \eqref{cond1} with $\alpha = 0$ is equivalent to: 

\begin{equation}
\begin{split}
P_{11} & = 0,P_{12}= -2 P_{23}+2 P_{31}-2 P_{33},P_{21}= P_{23}-P_{31}+P_{33}, \\
P_{22} & = P_{23}-P_{31}+P_{33},P_{13}= -2 P_{23}-2 P_{33},P_{32}= P_{33}
\end{split}
\end{equation}

Besides, by explicit calculating the determinant of $H_w$ (which is a polynomial of degree $2$ in $w$), it is easy to show that $P_{1,4}$ (coefficient of $w^2$) must vanish. Under that condition, we have:

\begin{equation}H_w = \left[
\begin{array}{ccc}
 -2 P_{23}+2 P_{31}-2 P_{33} & -2 P_{23}-2 P_{33} & 0 \\
 -w+P_{23}-P_{31}+P_{33} & P_{23} & P_{24} \\
 w+P_{33} & P_{33}-w & P_{34}
\end{array}
\right]
\end{equation}
and clearly $\det H_w$ is even i.e., $\det H_w \neq \pm 1$ what shows that there is no $P$ such that, for $\alpha = 0$, the conditions \eqref{cond1} and \eqref{cond2} simultaneously hold. However, following an analogous argument for $\alpha = 1$, we find the perturbation matrix:

\begin{equation}
P = \left[
\begin{array}{cccc}
 -1 & 1 & 1 & 0 \\
 0 & -1 & 0 & 1 \\
 0 & 0 & -1 & -2
\end{array}
\right].
\end{equation}
We extended this result for any $n$ through the following proposition:
\begin{proposition}
There is a sequence of projection lattices that converges to $D_n^*$ for any $n \geq 1$ at the rate $O(1/\left\|\bm{v}\right\|^{2/n})$. 
\end{proposition}
\begin{proof}
Here, the perurbation matrix is given by
\begin{equation}P_{ij} = \left\{\begin{array}{lc}
(-1)^{i} \displaystyle \genfrac{(}{)}{0pt}{}{n-1}{i-1} & \mbox{ if } j = n +1 \mbox{ and } i \geq 2 \mbox{,}\\
\,\, -1 & \mbox{ if } i = j, \\
\,\,\,\,\,\, 1 & \mbox{ if } i = 1 \mbox{ and } j \leq n, \\

\,\,\,\,\,\, 0 & \mbox{otherwise}
\end{array}. \right.\end{equation}
Again, by direct multiplication we can see that $G^*P^t + PG^{*t} = A^*$  ($\alpha = 1$) and by applying elementary operations to the matrix $H_w$ we see that $\det H_w = (-1)^{n+1}$ proving the statement. 
\end{proof}

For $n = 3$ the projection family associated to the vectors $$\bm{v_w} = \left[1, -2w^2+w+1, -2w^2-3w-2, 2w^3+3w^2+3w+1\right]$$ converges to the famous body-centered cubic lattice, which has the best covering density in three dimensions \cite{Sloane1}, at a rate $O(1/\left\|\bm{v}\right\|^{2/3}$).

\subsection{The lattice $E_8$}

For the lattice $E_8$, the problem of finding a perturbation matrix that speeds the convergence rate up to $O(1/w^2)$ has $72$ integer variables. After some simplifications (by explicitly solving equation \eqref{cond1}) we can reduce this problem to $36$ integer variables and $7$ non-linear restrictions, corresponding to the polynomial equality \eqref{cond2}, which has a high computational complexity. We do not know if there exists an exact solution to this problem. Hence, we generate sub-optimal solutions, considering the problem \eqref{eq:PNLIrel} and the Frobenius norm, which yields to a quadratic integer problem (IQP), and show that these solutions have good (exponential) gains in comparison to the sequences in \cite{Sloane1}. We also compare different equivalent integer representations for the $E_8$ lattice. It is worth reminding that $\left\| M \right\|_F \geq \left\| M \right\|_{\infty}$ for any matrix $M$.

The first representation for $\Lambda^* = E_8$ is the matrix \cite[p. 121]{SloaneBook}, the same as in \cite{Sloane1}. The second one is the matrix obtained by applying Construction A  \cite[ch.5]{SloaneBook} to the extended Hamming code  $\mathcal{H}(8,4)$ \cite[ch 3,sec 2.3]{SloaneBook}. We show the perturbation matrices found in both cases (respectively, $P_1$ and $P_2$) and compare the curves of center density versus the the logarithm to the basis 2 of the euclidean norm of the projection vector (Figure \ref{fig:E8}). For all these constructions, the rate of convergence of the produced sequences is $O(1/w)$. \\
\begin{equation*}
P_1 = 
\begin{bmatrix}
 0 & 1 & 0 & 0 & 0 & 0 & 0 & 0 & 0 \\
 0 & -1 & 1 & 0 & 0 & 0 & 0 & 0 & 0 \\
 0 & -1 & -1 & 1 & 0 & 0 & 0 & 0 & 0 \\
 0 & 1 & -1 & -1 & 1 & 0 & 0 & 0 & 0 \\
 0 & 0 & 1 & -1 & -1 & 1 & 0 & 0 & 0 \\
 0 & 0 & 0 & 1 & -1 & -1 & 1 & 0 & 0 \\
 0 & 0 & 0 & 0 & 1 & -1 & -1 & 1 & 0 \\
 0 & 0 & 0 & 0 & 0 & 0 & 0 & 0 & 1
\end{bmatrix},\end{equation*}
\begin{equation*}
 P_2 = \begin{bmatrix}
 0 & 1 & 0 & 0 & 0 & 0 & 0 & 0 & 0 \\
 0 & 0 & 1 & 0 & 0 & 0 & 0 & 0 & 0 \\
 0 & 0 & 0 & 1 & 0 & 0 & 0 & 0 & 0 \\
 0 & -1 & 0 & 1 & 1 & 0 & 0 & 0 & 0 \\
 0 & 1 & 0 & 0 & 0 & 1 & 0 & 0 & 0 \\
 0 & 0 & 1 & 1 & 0 & 0 & 1 & 0 & 0 \\
 0 & 0 & -1 & 1 & 0 & 0 & 0 & 1 & 0 \\
 0 & 0 & 0 & 0 & 1 & 1 & -1 & 0 & 1
\end{bmatrix}
\end{equation*}

\begin{figure}[htb!]

\includegraphics[scale=0.6]{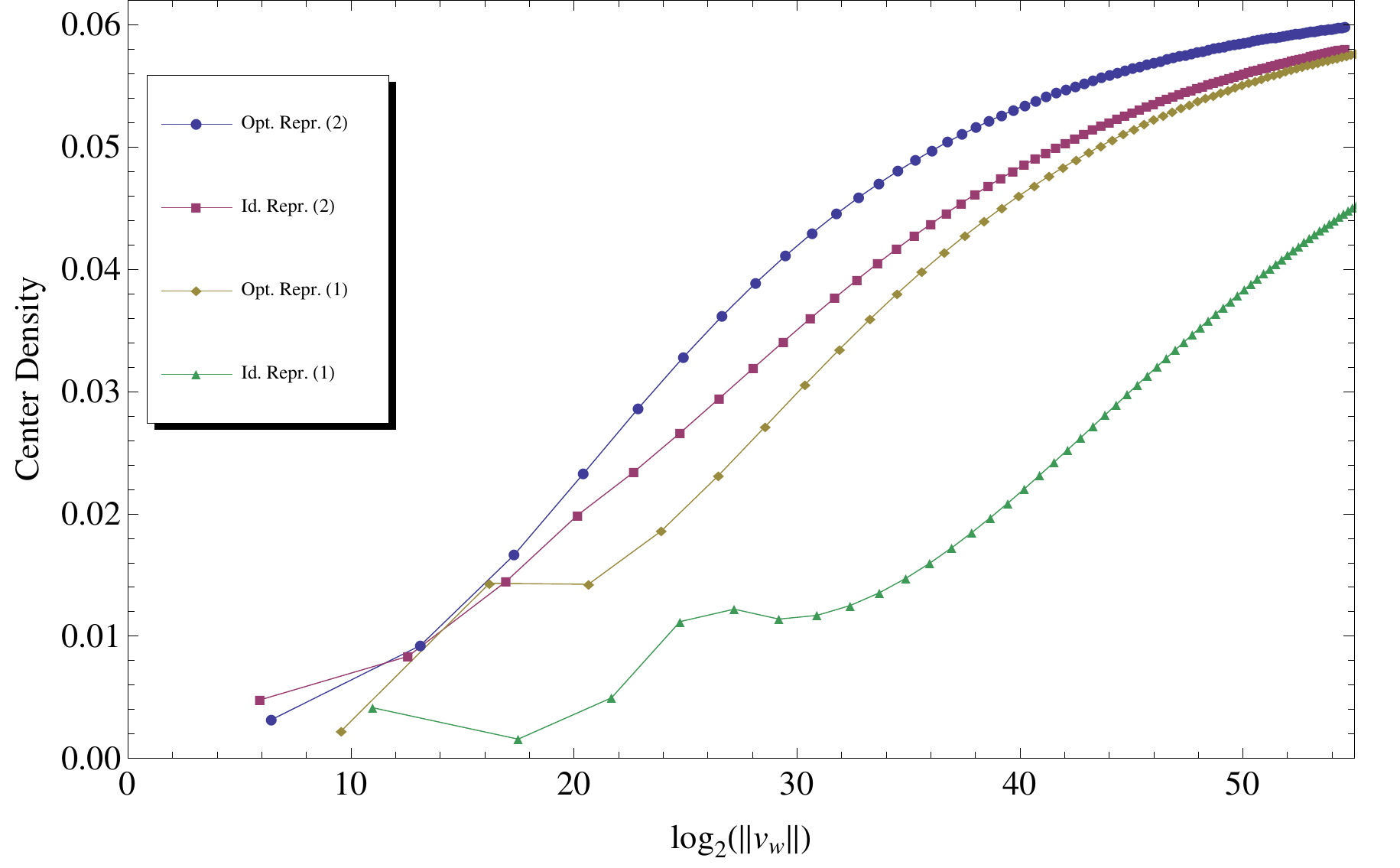}
\caption{The first two curves from the bottom to the top were obtained with the same generator matrix for $E_8$ as in \cite{Sloane1} and perturbation matrices $[0 \,\, I_8]$ and $P_1$, respectively. The last two correspond to representation $G_2$ above and perturbation matrices $[0 \,\, I_8]$ and $P_2$, respectively.}
\label{fig:E8}
\end{figure}

%

\section{Conclusion}
In this paper we address the problem of finding sequences of projection lattices with a good rate ``density versus length of the projection vector''. With a subtle modification of the Lifting Construction \cite{FatStrut}, we prove a sufficient condition for constructing projection lattices sequences that converge with order $O( 1/ \left\| \bm{v} \right\|^{2/n})$ to target lattices whose dual are integer. We then construct explicit examples of such sequences for some well-known lattices, such as $D_n$ (odd $n$) and $D_n^*$. We also show examples of good projection lattices sequences for the $E_8$ lattice that do not satisfy condition \eqref{cond2}. 

The question whether it is always possible to speed the convergence rate up to $O(1/\left\| \bm{v} \right\|^{2/n})$ remains open. Also explicit constructions for $D_n$ (even $n$) as well as other important lattices are let for further work. These constructions, however, seem to require totally different techniques. For instance, it is possible to verify computationally through an exhaustive search that there is no simultaneous solution for equations \eqref{cond1} and \eqref{cond2} for the lattice $D_4$.

We finish by exhibiting a table of which is, to our knowledge, the best projection lattices sequences in dimensions from $3$ to $8$, in the sense discussed in this paper (except $n = 6$)\footnote{To our knowledge, there is no integer representation in $\mathbb{R}^6$ for the best $6$-dimensional lattice packing $E_6$, hence the approach here cannot be employed. The projections of $\mathbb{Z}^7$ onto $\bm{v}^\perp$ for the vector presented here converge to $D_6$.}.

{\footnotesize 
\begin{table}[h]
\begin{center}
\begin{tabular}{| c || c | c |}
\hline
\textbf{n} ($\Lambda$) & \textbf{Vector}  \\
\hline
3 ($D_3$) & \footnotesize{ $( 1, 2 w^2+w+1, 2 w^2-w+1, -4 w^3-3 w )$} \\
\hline
4 ($D_4$)& \footnotesize{ $(1, -1+2 w, 4 w^2-2 w+1, 8 w^3-4 w^2+1, 8 w^4-8 w^3+4 w^2 )$} \\
\hline
\multirow{2}{*}{5 ($D_5$)} & \footnotesize{ $ (1, 4 w^3+2 w^2+3 w+1, -4 w^3+2 w^2-3 w+1$} \\ & \footnotesize{ $8 w^4+8 w^2+w+1, 8 w^4+8 w^2-w+1, 16 w^5+20 w^3+5w )$}
\\
\hline
\multirow{2}{*}{6 ($D_6$)} & \footnotesize{ $(1, 2 w-1, 4 w^2-2 w+1, 8 w^3-4 w^2+1, 16 w^4-8 w^3+4 w^2+1 $} \\ & \footnotesize{ $ 32 w^5-16 w^4+4 w^2+2 w-1, 32 w^6-32 w^5+16 w^4+2 w^2-2 w+1 )$ }
\\
\hline
\multirow{4}{*}{7 ($E_7$)} & \footnotesize{ $ (1, 1- 2 w, 4 w^2-4 w+2,-8 w^3+12 w^2-10 w+3, $} \\ & \footnotesize{ $8 w^4-16 w^3+18 w^2-10 w+2 $} \\ & \footnotesize{ $-8 w^5+16 w^4-30 w^3+28 w^2-16 w+4$} \\ & \footnotesize{ $8 w^6-16 w^5+38 w^4-44 w^3+36 w^2-16 w+3$} \\ & \footnotesize{ $ -8 w^7+16 w^6-46 w^5+60 w^4-70 w^3+50 w^2-24 w+5)$}
\\
\hline
\multirow{3}{*}{8 ($E_8$)} & \footnotesize{ $(
1, -2 w, 4 w^2, -8 w^3, 16 w^4+8 w^3-2 w, -16 w^5-8 w^4-4 w^3+4 w^2+w$} \\ & \footnotesize{ $16 w^6+8 w^5+12 w^4-3 w^2, -16 w^7-8 w^6-28 w^5+7 w^3+6 w^2$} \\ & \footnotesize{ $16 w^8+8 w^7+44 w^6+8 w^5-3 w^4-10 w^3-3 w^2+w
)$ }
\\
\hline
\end{tabular}

\caption{Best families of projection lattices of $\mathbb{Z}^{n+1}$ $n = 3$ to $8$ (except $n = 6$), converging to the target lattice $\Lambda$.} 
\end{center}
\end{table}}

\section*{Acknowledgments}
The authors would like to thank the anonymous referees for pointing out some mistakes in the first version of the paper. The authors also acknowledge Prof. Sueli Costa for fruitful discussions on projection lattices. This work was supported by S\~ao Paulo Research Foundation (FAPESP) under grant 2009/18337-6 for Antonio Campello and 2011/01096-6 for Jo\~ao Strapasson.

\bibliographystyle{plain}
\bibliography{versao_final}

\end{document}